\documentclass[11pt]{article}

\usepackage[margin=1in]{geometry}
\usepackage{amsmath,amssymb,amsthm}
\usepackage{mathtools}
\usepackage[hypertexnames=false]{hyperref}
\usepackage{url}
\usepackage[numbers,sort&compress]{natbib}
\usepackage{booktabs}
\usepackage{graphicx}
\usepackage{xcolor}
\usepackage{enumitem}
\usepackage{thmtools}

\newtheorem{theorem}{Theorem}[section]
\newtheorem{lemma}[theorem]{Lemma}
\newtheorem{proposition}[theorem]{Proposition}
\newtheorem{corollary}[theorem]{Corollary}

\theoremstyle{definition}
\newtheorem{definition}[theorem]{Definition}
\newtheorem{remark}[theorem]{Remark}

\newtheorem{openproblem}[theorem]{Open Problem}

\newcommand{\eps}{\varepsilon}

\newcommand{\lamH}{\lambda_1(A_H)}
\newcommand{\lamG}{\lambda_1(A_G)}
\newcommand{\AG}{A_G}
\newcommand{\AH}{A_H}
\newcommand{\LG}{L_G}
\newcommand{\LH}{L_H}
\newcommand{\be}{\mathbf{e}}
\newcommand{\bb}{\mathbf{b}}
\newcommand{\vone}{v_1}
\newcommand{\E}{\mathbb{E}}
\newcommand{\R}{\mathbb{R}}

\DeclareMathOperator{\Var}{Var}

\title{Adjacency Spectral Radius Under Laplacian Sparsification:\\
Deterministic and Probabilistic Bounds}

\author{Joshua Steier\\[4pt]
\small \texttt{joshsteier@gmail.com}\\[2pt]
\small \texttt{Independent Researcher}}

\date{}

\begin{document}
\maketitle

\begin{abstract}
Spielman--Srivastava spectral sparsification preserves Laplacian quadratic forms to
within $(1 \pm \eps)$, but does not directly control the adjacency spectral
radius~$\lamG$---which governs the NIMFA epidemic threshold and arises in spectral
clustering. We prove $|\lamH - \lamG| \leq \eps(2\Delta - \lamG)$ deterministically,
with a sharp $\eps\lamG$ bound for reweighting sparsifiers via Perron--Frobenius
monotonicity. Under effective-resistance sampling, Matrix Bernstein gives
$O(\eps\Delta/\sqrt{c})$ with high probability. Combining eigenvector delocalization
with resolvent perturbation theory, we establish that for graphs with delocalized
Perron eigenvectors and $\delta_{\mathrm{gap}} = \Omega(\Delta)$,
\[
|\lamH - \lamG| \;\leq\;
O\!\left(\frac{\eps\,\Delta\,\sqrt{\log n}}{\sqrt{n}}\right)
\;+\; O\!\left(\frac{\eps^2\Delta^2}{\delta_{\mathrm{gap}}}\right),
\]
with corollaries for Erd\H{o}s--R\'{e}nyi graphs, regular expanders, and stochastic
block models. Lower bounds establish tightness for regular graphs.
\end{abstract}


\section{Introduction}\label{sec:intro}

In computational epidemiology, large-scale simulations of SIS and SIR dynamics on contact
networks are often bottlenecked by graph size. Spielman--Srivastava spectral
sparsification~\citep{spielman2011graph} offers an attractive preprocessing step: it
samples edges proportional to their effective resistances, retaining only
$O(n \log n / \eps^2)$ edges while provably preserving all Laplacian quadratic forms
to within $(1 \pm \eps)$. Empirical studies confirm that epidemic dynamics on the
sparsified graph closely match the original~\citep{mercier2022effective}.

However, the theoretical guarantees of Laplacian sparsification do not directly address
the adjacency spectrum. In the $N$-intertwined mean-field approximation (NIMFA) of the
SIS model, the epidemic threshold is $\tau_{\mathrm{mf}}(G) = 1/\lamG$, where $\lamG$
is the spectral radius of the adjacency
matrix~\citep{vanmieghem2009virus,wang2003epidemic}. The quantity $1/\lamG$ is the
threshold of the mean-field approximation, not the exact sharp threshold of the
finite-state SIS Markov process, which is more subtle on general finite
graphs~\citep{vanmieghem2014exact}. Nevertheless, $1/\lamG$ is the standard spectral proxy
used throughout computational epidemiology and network science~\citep{vanmieghem2011nimfa}.

Since $A = D - L$, the adjacency spectrum depends on both the Laplacian and the degree
matrix, and the degree matrix is only approximately preserved under sparsification. This
raises a precise mathematical question:

\begin{quote}
\emph{If $H$ is a $(1 \pm \eps)$-Laplacian sparsifier of $G$, what is the tightest
provable bound on $|\lamH - \lamG|$?}
\end{quote}

\noindent To our knowledge, the specific question of bounding $|\lamH - \lamG|$ for
Laplacian-preserving sparsifiers has not been addressed in the literature, though the
surrounding space is not empty. That effective-resistance sparsification approximately
preserves the adjacency spectral radius---and hence the mean-field epidemic
threshold---has been observed empirically~\citep{mercier2022effective}, but we are not
aware of any \emph{provable, quantitative} bound on this distortion; supplying one is
the purpose of this work.
Random edge-sampling schemes such as DOULION~\citep{tsourakakis2009doulion}
retain each edge independently and reweight by the inverse sampling probability,
and have been used as a general-purpose sparsification primitive; our contribution
is to provide \emph{rigorous, quantitative} bounds on adjacency spectral radius
distortion specifically for \emph{Laplacian-preserving} sparsifiers, which offer
the additional structural guarantee of preserving all Laplacian quadratic forms.
The Spielman--Srivastava and
BSS~\citep{batson2012twice} papers bound Laplacian eigenvalues only.
Achlioptas--McSherry~\citep{achlioptas2007fast} and related work on matrix sparsification
use entry-wise sampling to approximate a matrix in spectral norm, but do not address the
graph-specific setting of Laplacian-preserving subgraphs. Kolla et
al.~\citep{kolla2010subgraph} study subgraph sparsification (retaining a subgraph of a
given graph union while preserving Laplacian quality), which is a different structural
question. Purohit et al.~\citep{purohit2014fast} minimize the change in $\lamG$ under
graph \emph{coarsening} (node merging), which changes the vertex set and is thus
incomparable with our edge-sparsification setting. Bravo Hermsdorff and
Gunderson~\citep{bravo2019unifying} provide a unifying framework for sparsification and
coarsening that preserves the Laplacian pseudoinverse, but do not analyze adjacency
eigenvalue distortion.

More recently, Jin, Karmarkar, Musco, Sidford, and Singh~\citep{jin2024spectrum}
develop nuclear sparsifiers that yield approximation guarantees for the
spectral density of the \emph{normalized} adjacency/Laplacian matrix in Wasserstein-1
distance, which is a distributional statement about all eigenvalues of a different
matrix operator. Cai, Chen,
and Peng~\citep{cai2023effective} prove leverage-score-based lower bounds on eigenvalues
after row/edge removal, which is structurally adjacent to our perturbation framing.
Neither addresses the top adjacency eigenvalue $\lamG$ under Laplacian-preserving
sparsification, which is the focus of this work.

This gap matters practically because many graph analysis pipelines sparsify a graph once
for computational efficiency (using effective resistance because it preserves diffusion
properties) but later analyze the adjacency spectrum for different purposes---mean-field
epidemic thresholds, spectral clustering, centrality estimation. Our results quantify the
collateral damage to $\lamG$ from this workflow.

\begin{remark}
Since the sparsifier $H$ is a weighted graph, the NIMFA threshold $1/\lamH$ assumes a
weighted SIS model where infection rates scale linearly with edge weights.
\end{remark}

\paragraph{Contributions.}
Our main result (Theorem~\ref{thm:delocalized}) shows that for graphs with
delocalized Perron eigenvectors and spectral gap $\delta_{\mathrm{gap}} = \Omega(\Delta)$,
Spielman--Srivastava sparsification preserves $\lamG$ to within
$O(\eps\sqrt{\Delta\log n}) + O(\eps^2\Delta)$ with high probability---far tighter
than the generic operator-norm bound $O(\eps\Delta)$.
The paper develops three layers of analysis leading to this result:

\begin{enumerate}[label=(\alph*),nosep]
\item \emph{Deterministic bounds.}
For any $(1\pm\eps)$-Laplacian sparsifier:
$|\lamH - \lamG| \leq \eps(2\Delta - \lamG)$ absolutely
(Theorem~\ref{thm:abs-det}), with a sharp $\eps$-relative bound
$|\lamH - \lamG| \leq \eps\lamG$ for reweighting sparsifiers via Perron--Frobenius
monotonicity (Theorem~\ref{thm:pf-sharp}). We introduce an eigenvector-degree
parameter $\gamma(G)$ that governs the one-sided bound
(Theorem~\ref{thm:gamma-lower}) and an effective degree $\alpha_w(G)$ that
isolates the true driver of distortion (Theorem~\ref{thm:refined-det}).

\item \emph{Probabilistic refinement via resolvent perturbation theory.}
Under SS sampling, a Matrix Bernstein bound gives
$\|\AH - \AG\|_2 \leq 2\eps\Delta/\sqrt{c}$ (Theorem~\ref{thm:bernstein}).
A second-order resolvent decomposition (Theorem~\ref{thm:resolvent})
combined with scalar concentration of the first-order term and eigenvector
delocalization yields the main two-term bound
(Theorem~\ref{thm:delocalized}), with explicit corollaries for
Erd\H{o}s--R\'{e}nyi graphs, $d$-regular expanders, and balanced stochastic
block models.

\item \emph{Lower bounds and tightness.}
Reweighting constructions establish that the deterministic bound is tight for
regular graphs and that $|\lamH - \lamG| \geq \eps\lamG$ for all graphs.
Whether edge \emph{deletion} can cause $\omega(\eps\lamG)$ distortion
remains open (Open Problem~\ref{op:genuine}).

\item \emph{Comparison with entry-wise sparsification.}
At equal edge count, the SS eigenvalue distortion for well-connected graphs
is a factor of $\sqrt{n/\log n}$ below what entry-wise sampling achieves,
while additionally preserving the Laplacian (Section~\ref{sec:comparison}).
\end{enumerate}

\paragraph{Organization.}
Section~\ref{sec:prelim} sets notation.
Section~\ref{sec:deterministic} proves the deterministic bounds and discusses
limitations of the $\gamma$-bound.
Section~\ref{sec:reweight} establishes the sharp Perron--Frobenius bound for
reweighting sparsifiers.
Section~\ref{sec:refined-det} gives the eigenvector-weighted deterministic refinement.
Section~\ref{sec:bernstein} proves the probabilistic operator-norm bound.
Section~\ref{sec:delocalized} establishes the delocalized distortion bound,
with corollaries for specific graph families and a table of resolvent conditions.
Section~\ref{sec:lower} establishes lower bounds and tightness.
Section~\ref{sec:comparison} compares with entry-wise sparsification.
Section~\ref{sec:experiments} presents computational validation.
Section~\ref{sec:discussion} discusses open problems.

\section{Preliminaries}\label{sec:prelim}

Let $G = (V, E)$ be a connected simple graph on $n = |V|$ vertices with
$m = |E|$ edges. We write $\AG$ for the adjacency matrix,
$D_G = \operatorname{diag}(d_1, \ldots, d_n)$ for the degree matrix, and
$\LG = D_G - \AG$ for the combinatorial Laplacian. Let $\Delta = \max_i d_i$
and $\delta = \min_i d_i$ denote the maximum and minimum degrees.

The eigenvalues of $\AG$ are $\lamG = \lambda_1 \geq \lambda_2 \geq \cdots \geq \lambda_n$.
We write $\vone$ for the unit eigenvector corresponding to $\lamG$
(i.e., $\AG \vone = \lamG \vone$, $\|\vone\|_2 = 1$). The \emph{spectral gap}
is $\delta_{\mathrm{gap}} = \lambda_1 - \lambda_2 > 0$ (strict positivity holds
for connected graphs by the Perron--Frobenius theorem).

For each edge $e = (u,v) \in E$, the \emph{effective resistance} is
$R_e = \bb_e^\top \LG^+ \bb_e$ where $\bb_e = \be_u - \be_v$ and $\LG^+$ is
the Moore--Penrose pseudoinverse. By Foster's theorem,
$\sum_{e \in E} R_e = n - 1$ for unweighted graphs.

We write $\Phi := \AH - \AG$ for the perturbation matrix.

\begin{definition}[Effective-resistance sampling scheme]\label{def:ss}
Given $G$ and a sample budget $q \geq 1$:
\begin{enumerate}[nosep]
\item Set sampling probabilities $p_e = R_e / (n-1)$ for each $e \in E$.
\item Draw $q$ edges i.i.d.\ from distribution $(p_e)_{e \in E}$.
\item For edge $e$ sampled $k_e$ times, assign weight $\tilde{w}_e = k_e / (q p_e)$ in $H$.
\end{enumerate}
\end{definition}

\begin{theorem}[Spielman--Srivastava {\citep{spielman2011graph}}]\label{thm:ss}
There exists an absolute constant $C > 0$ such that for any $\eps \in (0,1)$,
if $q \geq C \cdot n\log n / \eps^2$, then the output $H$ of
Definition~\ref{def:ss} satisfies
$(1-\eps)\LG \preceq \LH \preceq (1+\eps)\LG$ with high probability.
\end{theorem}

\begin{definition}[Eigenvector-degree parameter]\label{def:gamma}
For a connected graph $G$ with unit Perron eigenvector $\vone$, define
\[
\alpha(G) = \vone^\top D_G \vone = \sum_{i=1}^n d_i \, v_{1,i}^2,
\qquad
\gamma(G) = \frac{2\alpha(G)}{\lamG} - 1.
\]
\end{definition}

\begin{remark}\label{rem:gamma-bounds}
Since $\LG = D_G - \AG$ is positive semidefinite, $\vone^\top D_G \vone \geq
\vone^\top \AG \vone = \lamG$, so $\gamma(G) \geq 1$. For $d$-regular graphs,
$\vone = \mathbf{1}/\sqrt{n}$, $\alpha = d$, $\lamG = d$, so $\gamma = 1$.
For the star $K_{1,n-1}$, $\gamma = \Theta(\sqrt{n})$.
In general, $\gamma \leq 2\Delta/\lamG - 1$.
\end{remark}

\begin{definition}[Delocalization]\label{def:deloc}
A connected graph~$G$ is \emph{$K$-delocalized} if its unit Perron eigenvector satisfies
$\|\vone\|_\infty \leq K / \sqrt{n}$.
\end{definition}

\begin{remark}
For $d$-regular graphs, $\vone = \mathbf{1}/\sqrt{n}$, so $K = 1$.
For Erd\H{o}s--R\'{e}nyi $G(n,p)$ with $np \geq C\log n$,
the general eigenvector delocalization bound of~\citet{he2019eigenvectors}
gives $\|\vone\|_\infty \leq n^{-1/2+\eps}$ for any $\eps > 0$ (i.e.,
$K = n^{o(1)}$). For the Perron eigenvector specifically, the entrywise
perturbation theory of~\citet{abbe2020entrywise} yields the stronger
$K = O(1)$ in the regime $np/\log n \to \infty$; see
Corollary~\ref{cor:er} for details.
\end{remark}

\section{Deterministic Bounds}\label{sec:deterministic}

\begin{lemma}[Degree preservation]\label{lem:degree}
If $(1-\eps)\LG \preceq \LH \preceq (1+\eps)\LG$, then for every vertex $i$:
\[
(1-\eps)\, d_i(G) \leq d_i(H) \leq (1+\eps)\, d_i(G).
\]
\end{lemma}
\begin{proof}
Evaluate the Laplacian quadratic form at $x = \be_i$:
$x^\top L x = L_{ii} = d_i$ for any graph with no self-loops.
The two-sided Laplacian guarantee then gives the result directly.
\end{proof}

\begin{theorem}[Deterministic $\gamma$-bound, lower]\label{thm:gamma-lower}
Let $H$ be a $(1\pm\eps)$-Laplacian sparsifier of $G$. Then
\[
\lamH \geq \lamG\!\left[1 - \eps\,\gamma(G)\right].
\]
\end{theorem}
\begin{proof}
From $\LH \preceq (1+\eps)\LG$ we get
$\AH = D_H - \LH \succeq D_H - (1+\eps)(D_G - \AG)$.
Using Lemma~\ref{lem:degree} ($D_H \succeq (1-\eps)D_G$):
\begin{align*}
\vone^\top \AH \vone
&\geq (1-\eps)\,\vone^\top D_G \vone - (1+\eps)\,\vone^\top D_G \vone + (1+\eps)\lamG \\
&= -2\eps\,\alpha(G) + (1+\eps)\lamG
= \lamG\!\left[1 - \eps\,\gamma(G)\right].
\end{align*}
Since $\lamH \geq \vone^\top \AH \vone$ by the Rayleigh quotient, the result follows.
\end{proof}

\begin{theorem}[Deterministic upper bound]\label{thm:det-upper}
Let $H$ be a $(1\pm\eps)$-Laplacian sparsifier of $G$. Then
\[
\lamH \leq (1-\eps)\lamG + 2\eps\Delta.
\]
\end{theorem}
\begin{proof}
From $\LH \succeq (1-\eps)\LG$:
$\AH \preceq D_H - (1-\eps)\LG = D_H - (1-\eps)(D_G - \AG)$.
Using $D_H \preceq (1+\eps)D_G$: for any unit vector $u$,
\begin{align*}
u^\top \AH u
&\leq (1+\eps)\,u^\top D_G u - (1-\eps)\,u^\top D_G u + (1-\eps)\,u^\top \AG u \\
&= 2\eps\,u^\top D_G u + (1-\eps)\,u^\top \AG u
\leq 2\eps\Delta + (1-\eps)\lamG.
\end{align*}
Taking $u$ to be the top eigenvector of $\AH$ and noting
$\lamH = u^\top \AH u$ completes the proof.
\end{proof}

\begin{theorem}[Deterministic absolute bound]\label{thm:abs-det}
Let $H$ be a $(1\pm\eps)$-Laplacian sparsifier of $G$. Then
\[
|\lamH - \lamG| \leq \eps(2\Delta - \lamG) \leq 2\eps\Delta.
\]
\end{theorem}
\begin{proof}
This combines the two one-sided bounds.
From Theorem~\ref{thm:gamma-lower}, using $\alpha(G) \leq \Delta$:
\[
\lamH \geq \lamG[1 - \eps\gamma(G)]
= \lamG + \eps\lamG - 2\eps\alpha(G)
\geq \lamG - \eps(2\Delta - \lamG).
\]
From Theorem~\ref{thm:det-upper}:
\[
\lamH \leq (1-\eps)\lamG + 2\eps\Delta = \lamG + \eps(2\Delta - \lamG).
\]
Combining gives $|\lamH - \lamG| \leq \eps(2\Delta - \lamG)$.
The weaker bound $\eps(2\Delta - \lamG) \leq 2\eps\Delta$ follows from $\lamG \geq 0$.
\end{proof}

\begin{corollary}[Deterministic relative bound]\label{cor:relative}
Let $H$ be a $(1\pm\eps)$-Laplacian sparsifier of a connected graph $G$. Then:
\[
\frac{|\lamH - \lamG|}{\lamG} \leq \eps\!\left(\frac{2\Delta}{\lamG} - 1\right)
\leq \eps\!\left(2\sqrt{\Delta} - 1\right).
\]
\end{corollary}
\begin{proof}
The first inequality is immediate from Theorem~\ref{thm:abs-det}.
For the second, we use the classical bound $\lamG \geq \sqrt{\Delta}$, which holds
because the neighborhood of any maximum-degree vertex $v$ spans at least $\Delta$
edges, so $G$ contains the star $K_{1,\Delta}$ as a subgraph (after identifying $v$
with the hub and its neighbors with the leaves). Embedding $K_{1,\Delta}$ into $G$
shows that $(\AG)_{ij} \geq (A_{K_{1,\Delta}})_{ij}$ entry-wise (in the principal
submatrix indexed by $\{v\} \cup N(v)$, with zeros elsewhere). Since $\AG$ is
nonneg\-ative and irreducible, the monotonicity of the Perron root under entry-wise
domination~\citep[Corollary~8.1.19]{horn2012matrix} gives
$\lamG \geq \lambda_1(K_{1,\Delta}) = \sqrt{\Delta}$.
\end{proof}

\begin{corollary}[NIMFA threshold preservation]\label{cor:epidemic}
Let $H$ be a $(1\pm\eps)$-Laplacian sparsifier of $G$, and let
$\tau_{\mathrm{mf}}(G) = 1/\lamG$ be the NIMFA epidemic threshold.
If $\eps\,(2\Delta/\lamG - 1) < 1$, then
\[
\frac{1}{1 + \eps(2\Delta/\lamG - 1)}
\leq \frac{\tau_{\mathrm{mf}}(H)}{\tau_{\mathrm{mf}}(G)}
\leq \frac{1}{1 - \eps\,\gamma(G)}.
\]
\end{corollary}

The lower bound in Corollary~\ref{cor:epidemic} uses only $\Delta$ and $\lamG$ (both
computable in $O(m)$ and $O(m\log n)$ time respectively). The tighter upper bound
uses $\gamma(G) = 2\alpha(G)/\lamG - 1$, which requires the Perron eigenvector; in
practice, a few steps of power iteration on $\AG$ suffice for an accurate estimate.

\begin{proposition}[Jensen bias]\label{thm:jensen}
Let $H$ be the random output of the effective-resistance sampling scheme
(Definition~\ref{def:ss}). Then $\E[\lamH] \geq \lamG$.
\end{proposition}
\begin{proof}
The map $A \mapsto \lambda_1(A)$ is convex (supremum of linear functions).
Since $\E[\AH] = \AG$ (unbiasedness of the SS estimator), Jensen's inequality gives
$\E[\lamH] \geq \lambda_1(\E[\AH]) = \lamG$.
\end{proof}

\subsection{Limitations of $\gamma(G)$ for Heterogeneous Graphs}

For heterogeneous graphs, $\gamma(G)$ can be large, making the $\gamma$-bound vacuous.

\begin{proposition}[Star graph $\gamma$]\label{prop:star-gamma}
For the star $K_{1,n-1}$,
$\gamma(K_{1,n-1}) = n/\sqrt{n-1} - 1 = \Theta(\sqrt{\Delta})$.
\end{proposition}
\begin{proof}
For $K_{1,n-1}$: $\lamG = \sqrt{n-1}$, $v_{1,\mathrm{hub}} = 1/\sqrt{2}$,
$v_{1,\mathrm{leaf}} = 1/\sqrt{2(n-1)}$. Then
$\alpha = (n-1) \cdot \tfrac{1}{2} + (n-1) \cdot \tfrac{1}{2(n-1)} = n/2$,
giving $\gamma = n/\sqrt{n-1} - 1 = \Theta(\sqrt{\Delta})$.
\end{proof}

When $\gamma(G) = \Theta(\sqrt{\Delta})$, the $\gamma$-bound is vacuous for moderate
$\eps$; Corollary~\ref{cor:relative} still provides $O(\eps\sqrt{\Delta})$ relative
distortion unconditionally. For power-law graphs with degree exponent
$\beta \in (5/2, 3)$~\citep{chung2003eigenvalues}, $\gamma$ can reach
$\Theta(\sqrt{\Delta})$; extending the probabilistic analysis to such localized
eigenvector regimes is an important open direction
(Section~\ref{sec:discussion}).

\section{Sharp Bound for Reweighting Sparsifiers}\label{sec:reweight}

When the sparsifier retains all edges of~$G$ (with modified weights), the adjacency
perturbation can be bounded much more tightly than the generic Laplacian-based
approach. This yields a relative distortion of exactly~$\eps$, independent of~$\Delta$.

\begin{definition}[Reweighting sparsifier]\label{def:reweight}
A \emph{$(1\pm\eps)$-reweighting sparsifier} of $G$ is a weighted graph $H$ on the same
vertex and edge set as $G$, with edge weights $w_e \in [(1-\eps), (1+\eps)]$ for every
$e \in E$, such that $(1-\eps)\LG \preceq \LH \preceq (1+\eps)\LG$.
\end{definition}

\begin{theorem}[Perron--Frobenius relative bound]\label{thm:pf-sharp}
Let $H$ be a $(1\pm\eps)$-reweighting sparsifier of a connected graph~$G$
(Definition~\ref{def:reweight}). Then
\[
(1-\eps)\,\lamG \;\leq\; \lamH \;\leq\; (1+\eps)\,\lamG.
\]
In particular, $|\lamH - \lamG| \leq \eps\,\lamG$.
\end{theorem}

\begin{proof}
Since $H$ retains every edge of $G$ with weight $w_e \in [(1-\eps),(1+\eps)]$, the
adjacency matrices satisfy the entry-wise inequalities
\[
(1-\eps)\,(\AG)_{ij} \;\leq\; (\AH)_{ij} \;\leq\; (1+\eps)\,(\AG)_{ij}
\qquad \forall\; i,j.
\]
Both $\AG$ and $\AH$ are nonneg\-ative and irreducible (since $G$ is connected and all
edge weights are positive). By the monotonicity of the Perron root for irreducible
nonneg\-ative matrices (see, e.g., \citet{horn2012matrix}, Corollary~8.1.19): if
$0 \leq B \leq C$ entry-wise with $C$ irreducible, then $\rho(B) \leq \rho(C)$, with
equality if and only if $B = C$.

Applying this to $\AH \leq (1+\eps)\AG$:
\[
\lamH = \rho(\AH) \leq \rho\!\bigl((1+\eps)\AG\bigr) = (1+\eps)\,\lamG.
\]
Applying this to $(1-\eps)\AG \leq \AH$:
\[
(1-\eps)\,\lamG = \rho\!\bigl((1-\eps)\AG\bigr) \leq \rho(\AH) = \lamH.
\]
Combining gives the stated two-sided bound.
\end{proof}

A natural class of reweighting sparsifiers arises from Laplacian-preserving
modifications of tree substructures.

\begin{proposition}[Trees: Laplacian $\Leftrightarrow$ per-edge bounds]\label{prop:tree}
Let $T$ be a tree on $n$ vertices. A weighted graph~$H$ on the same vertex and edge set
as~$T$ with edge weights $(w_e)_{e \in E(T)}$ satisfies
$(1-\eps)L_T \preceq L_H \preceq (1+\eps)L_T$ if and only if
$w_e \in [(1-\eps), (1+\eps)]$ for every edge~$e$.
\end{proposition}

\begin{proof}
For a tree, the Laplacian quadratic form decomposes as
$x^\top L_H x = \sum_{e=(u,v)} w_e (x_u - x_v)^2$.
Each term $(x_u - x_v)^2$ is independently controllable: for any
$e_0 = (u_0, v_0) \in E(T)$, one can construct $x \in \R^n$ with
$x_{u_0} \neq x_{v_0}$ and $x_u = x_v$ for all other edges $e = (u,v) \neq e_0$
(this is possible precisely because $T$ is a tree, so removing $e_0$ disconnects $T$
into two components; set $x$ to different constants on each component).
For such~$x$, the Laplacian ratio reduces to $w_{e_0}$. The forward direction is
immediate; the reverse follows from the decomposition.
\end{proof}

\begin{corollary}\label{cor:tree-sharp}
For any tree $T$ and any $(1\pm\eps)$-Laplacian sparsifier $H$ that retains all edges of
$T$, we have $|\lambda_1(A_H) - \lambda_1(A_T)| \leq \eps\,\lambda_1(A_T)$.
\end{corollary}

\begin{proof}
Combine Proposition~\ref{prop:tree} with Theorem~\ref{thm:pf-sharp}.
\end{proof}

\section{Refined Deterministic Bound via Quadratic Form Decomposition}%
\label{sec:refined-det}

The deterministic bound $|\lamH - \lamG| \leq \eps(2\Delta - \lamG)$ arises from
bounding $D_H - D_G$ and $L_H - L_G$ independently. We now prove a refined bound
that exploits the correlation between these perturbations.

\begin{theorem}[Eigenvector-weighted deterministic bound]\label{thm:refined-det}
Let $H$ be a $(1\pm\eps)$-Laplacian sparsifier of~$G$. Let $\vone$ be the unit Perron
eigenvector of $\AG$ and let $w$ be the unit eigenvector of $\AH$ achieving $\lamH$.
Then:
\begin{enumerate}[label=(\alph*),nosep]
\item $\lamH - \lamG \leq 2\eps\,\alpha_w(G)$,
where $\alpha_w(G) := w^\top D_G w = \sum_i d_i w_i^2$.
\item $\lamG - \lamH \leq 2\eps\,\alpha(G) = 2\eps\,\vone^\top D_G\vone$.
\end{enumerate}
The quantity $\alpha_w(G) \leq \Delta$ always, but for delocalized eigenvectors
$\alpha_w(G) = O(\Delta \cdot K^2)$ independently of $n$.
\end{theorem}

\begin{proof}
\emph{Part (a).} From the proof of Theorem~\ref{thm:det-upper}:
\[
w^\top \AH w \leq 2\eps\,w^\top D_G w + (1-\eps)\,w^\top \AG w
\leq 2\eps\,\alpha_w(G) + (1-\eps)\lamG,
\]
so $\lamH = w^\top \AH w \leq \lamG + 2\eps\,\alpha_w(G) - \eps\lamG
= \lamG + \eps(2\alpha_w(G) - \lamG)$.
Since $\eps\lamG \geq 0$, we obtain the cleaner bound
$\lamH - \lamG \leq 2\eps\alpha_w(G)$.

\emph{Part (b).} From the proof of Theorem~\ref{thm:gamma-lower}:
\[
\vone^\top \AH \vone \geq -2\eps\alpha(G) + (1+\eps)\lamG
= \lamG - \eps(2\alpha(G) - \lamG).
\]
So $\lamH \geq \lamG - 2\eps\alpha(G) + \eps\lamG \geq \lamG - 2\eps\alpha(G)$.
\end{proof}

\section{Probabilistic Bounds via Matrix Concentration}\label{sec:bernstein}

The deterministic bound $|\lamH - \lamG| \leq \eps(2\Delta - \lamG)$ treats the
perturbation $\AH - \AG$ as adversarial. When $H$ is constructed via SS sampling,
we can exploit the randomness to obtain high-probability bounds.

\subsection{Matrix Bernstein Operator Norm Bound}

\begin{lemma}[Effective resistance bounds for edges]\label{lem:re-bounds}
Let $G$ be a connected unweighted graph with maximum degree~$\Delta$.
For every edge $e = (u,v) \in E$:
\[
\frac{1}{\min(d_u, d_v)} \;\leq\; R_e \;\leq\; 1.
\]
In particular, $1/\Delta \leq R_e \leq 1$.
\end{lemma}

\begin{proof}
\emph{Upper bound.}
The effective resistance $R_e$ equals the minimum energy of a unit $(u,v)$-flow.
Routing one unit of flow along the single edge~$e$ (which has unit conductance) gives
a feasible flow with energy~$1$. Therefore $R_e \leq 1$.

\emph{Lower bound.}
Apply a unit potential difference across $e$: set the voltage at $u$ to~$1$ and
the voltage at~$v$ to~$0$. By the maximum principle for harmonic functions on graphs,
every other vertex $w \in V \setminus \{u,v\}$ has voltage $V(w) \in (0,1)$.
By Ohm's law (each edge has unit resistance), the current flowing out of~$u$
along edge $(u,w)$ is $V(u) - V(w) = 1 - V(w)$. Since $V(w) \in (0,1)$ for every
neighbor~$w$ of~$u$, each such current is positive and at most~$1$. The total
current out of~$u$ (which equals $1/R_e$ by definition of effective resistance) is
\[
\frac{1}{R_e}
= \sum_{w \in N(u)} \bigl(1 - V(w)\bigr)
\leq \sum_{w \in N(u)} 1
= d_u.
\]
Therefore $R_e \geq 1/d_u$. By symmetry (applying the same argument at~$v$
with voltages reversed), $R_e \geq 1/d_v$. Combining:
$R_e \geq 1/\min(d_u, d_v) \geq 1/\Delta$.
\end{proof}

\begin{theorem}[Universal operator norm bound]\label{thm:bernstein}
Let $H$ be the output of the effective-resistance sampling scheme
(Definition~\ref{def:ss}) with
$q = \lceil cn\log n/\eps^2 \rceil$ samples, $c \geq 4$, and $\eps \in (0, 1/2]$.
Then with probability at least $1 - 2n^{-1/2}$:
\[
\|\AH - \AG\|_2 \leq \frac{2\eps\Delta}{\sqrt{c}}.
\]
\end{theorem}

\begin{proof}
\textbf{Step 1: Independent sum decomposition.}
Each of $q$ i.i.d.\ draws selects edge $e_k = (u_k, v_k)$ with probability
$p_{e_k} = R_{e_k}/(n-1)$. Define
\[
Y_k = \frac{1}{q\,p_{e_k}}\,P_{e_k}, \qquad
P_e = \be_u \be_v^\top + \be_v \be_u^\top.
\]
Then $\AH = \sum_{k=1}^q Y_k$ and $\E[Y_k] = \AG/q$ (by unbiasedness).
Setting $X_k = Y_k - \E[Y_k]$:
\[
\AH - \AG = \sum_{k=1}^q X_k, \qquad \E[X_k] = 0, \quad X_1, \ldots, X_q \text{ i.i.d.}
\]

\textbf{Step 2: Almost-sure bound $B$.}
By Lemma~\ref{lem:re-bounds}, $1/\Delta \leq R_e \leq 1$ for every edge~$e$.

From these bounds: $p_e = R_e/(n-1) \leq 1/(n-1)$ and
$p_e \geq 1/((n-1)\Delta)$, so $p_{\min} \geq 1/((n-1)\Delta)$.
Define $B := 2(n-1)\Delta/q$.
Since $\|P_e\|_2 = 1$, we have $\|Y_k\|_2 = 1/(qp_{e_k})$.
By the triangle inequality,
\[
\|X_k\|_2 \leq \|Y_k\|_2 + \|\E[Y_k]\|_2
= \frac{1}{qp_{e_k}} + \frac{\lamG}{q}.
\]
Now $1/(qp_{e_k}) \leq (n-1)\Delta/q$ and $\lamG/q \leq \Delta/q \leq (n-1)\Delta/q$
(using $\lamG \leq \Delta$ and $n - 1 \geq 1$), so
\[
\|X_k\|_2 \leq \frac{(n-1)\Delta}{q} + \frac{(n-1)\Delta}{q} = B.
\]

\textbf{Step 3: Variance proxy $v$.}
Since $X_k = Y_k - \E[Y_k]$, we have
$\E[X_1^2] \preceq \E[Y_1^2]$.
Using $P_e^2 = \be_u\be_u^\top + \be_v\be_v^\top$:
\[
\E[Y_1^2] = \frac{1}{q^2}\sum_e \frac{1}{p_e}(\be_u\be_u^\top + \be_v\be_v^\top).
\]
This is diagonal with $(i,i)$ entry $(n-1)/(q^2) \cdot \sum_{e \ni i} 1/R_e$.
Using $R_e \geq 1/\Delta$: $\sum_{e \ni i} 1/R_e \leq d_i \cdot \Delta \leq \Delta^2$, so
\[
\sigma^2 := q\,\|\E[X_1^2]\|_2 \leq v := \frac{(n-1)\Delta^2}{q}.
\]
Substituting $q = \lceil cn\log n/\eps^2 \rceil \geq cn\log n/\eps^2$:
\begin{equation}\label{eq:v-bound}
v \leq \frac{(n-1)\eps^2\Delta^2}{cn\log n}
< \frac{\eps^2\Delta^2}{c\log n}.
\end{equation}

\textbf{Step 4: Matrix Bernstein application.}
By the Matrix Bernstein inequality~\citep{tropp2012user}: for i.i.d.\ mean-zero
symmetric $n \times n$ matrices with $\|X_k\|_2 \leq B$ a.s.,
\[
\Pr\!\left[\left\|\textstyle\sum_k X_k\right\|_2 \geq t\right]
\leq 2n\exp\!\left(-\frac{t^2/2}{\sigma^2 + Bt/3}\right).
\]
Set $t = 2\eps\Delta/\sqrt{c}$. We verify the two key quantities.

\medskip\noindent\emph{Ratio $t^2/(2v)$.}
\[
\frac{t^2}{2v}
= \frac{4\eps^2\Delta^2/c}{2\eps^2\Delta^2/(c\log n)}
= 2\log n.
\]

\medskip\noindent\emph{Ratio $Bt/3$ vs.\ $v$.}
Recall $B = 2(n-1)\Delta/q$ and $v = (n-1)\Delta^2/q$, so
\[
\frac{Bt}{3v}
= \frac{2(n-1)\Delta \cdot 2\eps\Delta/(q\sqrt{c})}{3(n-1)\Delta^2/q}
= \frac{4\eps}{3\sqrt{c}}.
\]
For $\eps \leq 1/2$ and $c \geq 4$: $Bt/(3v) \leq 4/(3 \cdot 2 \cdot 2) = 1/3$.
Therefore $\sigma^2 + Bt/3 \leq v + v/3 = (4/3)v$, and
\[
\frac{t^2/2}{\sigma^2 + Bt/3}
\geq \frac{t^2/2}{(4/3)v}
= \frac{3}{4} \cdot \frac{t^2}{2v}
= \frac{3}{4} \cdot 2\log n
= \frac{3}{2}\log n.
\]
The failure probability is therefore at most
\[
2n \exp\!\left(-\tfrac{3}{2}\log n\right)
= 2n \cdot n^{-3/2} = 2n^{-1/2}. \qedhere
\]
\end{proof}

\begin{corollary}[Eigenvalue distortion bound]\label{cor:weyl}
Under the hypotheses of Theorem~\ref{thm:bernstein}, with probability at least
$1 - 2n^{-1/2}$:
\[
|\lamH - \lamG| \leq \frac{2\eps\Delta}{\sqrt{c}}.
\]
\end{corollary}
\begin{proof}
By Weyl's inequality, $|\lamH - \lamG| \leq \|\AH - \AG\|_2$.
\end{proof}

\subsection{Second-Order Resolvent Remainder Bound}

When the spectral gap $\delta_{\mathrm{gap}} = \lambda_1 - \lambda_2$ is large
relative to the perturbation, we can refine the bound using the known eigenvector $\vone$.

\begin{theorem}[Resolvent remainder bound]\label{thm:resolvent}
Let $\Phi = \AH - \AG$ and suppose $\|\Phi\|_2 < \delta_{\mathrm{gap}}/2$. Then:
\[
0 \leq (\lamH - \lamG) - \vone^\top \Phi\, \vone
\leq \frac{\|\Phi\vone\|^2}{\delta_{\mathrm{gap}} - 2\|\Phi\|_2}.
\]
In particular, $\lamH \geq \lamG + \vone^\top \Phi\vone$ always.

The condition $\|\Phi\|_2 < \delta_{\mathrm{gap}}/2$ holds with probability at least
$1 - 2n^{-1/2}$ (by Theorem~\ref{thm:bernstein}) whenever
$4\eps\Delta/\sqrt{c} < \delta_{\mathrm{gap}}$.
\end{theorem}

\begin{proof}
Write $\AH = \AG + \Phi$ in the eigenbasis $\{\vone, v_2, \ldots, v_n\}$ of $\AG$.
The eigenvalue equation $(\AG + \Phi)u = \lamH u$ decomposes into
a $1 \times 1$ block coupled to the $(n-1)$-dimensional
orthogonal complement. Eliminating the complement via the Schur complement gives:
\[
\lamH = \lambda_1 + \Phi_{11} + \mathbf{e}^\top\!
\left(\lamH I - \Lambda_\perp - \Phi_\perp\right)^{-1}\!\mathbf{e}
\]
where $\Phi_{11} = \vone^\top \Phi \vone$,
$\mathbf{e} = (v_2^\top \Phi\vone, \ldots, v_n^\top \Phi\vone)^\top$,
$\Lambda_\perp = \mathrm{diag}(\lambda_2, \ldots, \lambda_n)$,
and $\Phi_\perp$ is the restriction of $\Phi$ to the orthogonal complement
(cf.~\citet{stewart1990matrix}, Chapter~V).

When $\|\Phi\|_2 < \delta_{\mathrm{gap}}/2$:
$\lamH \geq \lambda_1 - \|\Phi\|_2 > \lambda_2 + \|\Phi\|_2 \geq \lambda_j + \|\Phi_\perp\|_2$
for all $j \geq 2$. The resolvent is therefore positive definite with minimum
eigenvalue $\geq \delta_{\mathrm{gap}} - 2\|\Phi\|_2 > 0$.
Thus:
\[
0 \leq \lamH - \lambda_1 - \Phi_{11}
\leq \frac{\|\mathbf{e}\|^2}{\delta_{\mathrm{gap}} - 2\|\Phi\|_2}
\leq \frac{\|\Phi\vone\|^2}{\delta_{\mathrm{gap}} - 2\|\Phi\|_2}
\]
where the last step uses $\|\mathbf{e}\|^2 = \|\Phi\vone\|^2 - \Phi_{11}^2 \leq \|\Phi\vone\|^2$.
\end{proof}

\begin{proposition}[Expected perturbation variance]\label{prop:variance}
Under SS sampling with $q = \lceil cn\log n/\eps^2 \rceil$:
\[
\E\!\left[\|\Phi\vone\|^2\right] \leq \frac{\eps^2\Delta^2}{c\log n}.
\]
\end{proposition}
\begin{proof}
Since $\Phi = \sum_{k=1}^q X_k$ with $X_k$ i.i.d.\ and mean-zero,
$\E[\Phi^\top \Phi] = q\,\E[X_1^2]$. Therefore
$\E[\|\Phi\vone\|^2] = \vone^\top \E[\Phi^\top \Phi]\vone
= q \cdot \vone^\top \E[X_1^2]\vone
\leq q\,\|\E[X_1^2]\|_2 \leq (n-1)\Delta^2/q
< \eps^2\Delta^2/(c\log n)$.
\end{proof}

\section{Distortion Bound Under Delocalization}\label{sec:delocalized}

We now prove a refined distortion bound
for graphs satisfying two structural hypotheses: eigenvector delocalization and a
sufficiently large spectral gap. These conditions hold for dense Erd\H{o}s--R\'{e}nyi
graphs, $d$-regular expanders, and balanced stochastic block models.

\subsection{Scalar Concentration of the First-Order Term}

\begin{lemma}[First-order perturbation concentration]\label{lem:first-order}
Let $G$ be $K$-delocalized and let $H$ be the SS sparsifier with
$q = \lceil cn\log n/\eps^2 \rceil$, $c \geq 4$, $\eps \in (0, 1/2]$.
There exists a threshold $n_0 = n_0(K, c)$ such that for all $n \geq n_0$,
with probability at least $1 - 2n^{-3/2}$:
\[
|\vone^\top \Phi\,\vone|
\;\leq\; \frac{4K\,\eps\,\Delta\,\sqrt{\log n}}{\sqrt{cn}}.
\]
\end{lemma}

\begin{proof}
Write $\Phi = \sum_{k=1}^q X_k$ with $X_k$ i.i.d.\ mean-zero.
Define the scalar random variables $Z_k := \vone^\top X_k \vone$. These are
i.i.d.\ with $\E[Z_k] = 0$, and $\vone^\top \Phi\,\vone = \sum_{k=1}^q Z_k$.

\medskip\noindent\textbf{Almost-sure bound.}
Since $X_k = Y_k - \E[Y_k]$ with $Y_k = P_{e_k}/(q\,p_{e_k})$, and
$\vone^\top P_e \vone = 2v_{1,u}v_{1,v}$ for edge $e = (u,v)$:
\[
|Z_k| \leq |\vone^\top Y_k \vone| + |\vone^\top \E[Y_k] \vone|
= \frac{2|v_{1,u_k}\,v_{1,v_k}|}{q\,p_{e_k}} + \frac{\lamG}{q}.
\]
Using $\|\vone\|_\infty \leq K/\sqrt{n}$, $1/p_{e_k} \leq (n-1)\Delta$
(since $R_{e_k} \geq 1/\Delta$ gives $p_{e_k} \geq 1/((n-1)\Delta)$), and
$\lamG \leq \Delta$:
\[
|Z_k| \leq \frac{2K^2(n-1)\Delta}{nq} + \frac{\Delta}{q}
\leq \frac{(2K^2+1)\Delta}{q} =: b.
\]

\medskip\noindent\textbf{Variance.}
\begin{align*}
\Var\!\left(\sum_k Z_k\right) &= q \cdot \E[Z_1^2] \leq q \cdot \E[(\vone^\top Y_1 \vone)^2].
\end{align*}
Expanding:
\[
\E[(\vone^\top Y_1 \vone)^2]
= \frac{1}{q^2} \sum_e \frac{(2v_{1,u}v_{1,v})^2}{p_e}
= \frac{4(n-1)}{q^2} \sum_{e=(u,v)} \frac{v_{1,u}^2\,v_{1,v}^2}{R_e}.
\]
Using $R_e \geq 1/\Delta$ and the estimate
$v_{1,u}^2 v_{1,v}^2 \leq \|\vone\|_\infty^2 \cdot (v_{1,u}^2 + v_{1,v}^2)/2
\leq (K^2/(2n))(v_{1,u}^2 + v_{1,v}^2)$:
\begin{align*}
\sum_e \frac{v_{1,u}^2 v_{1,v}^2}{R_e}
&\leq \frac{K^2\Delta}{2n}\sum_e (v_{1,u}^2 + v_{1,v}^2)
= \frac{K^2\Delta}{2n}\sum_i d_i\,v_{1,i}^2
= \frac{K^2\Delta\,\alpha(G)}{2n}
\leq \frac{K^2\Delta^2}{2n}.
\end{align*}
Therefore:
\[
\sigma^2 := \Var\!\left(\sum_k Z_k\right)
\leq q \cdot \frac{4(n-1)}{q^2} \cdot \frac{K^2\Delta^2}{2n}
= \frac{2K^2(n-1)\Delta^2}{qn}
\leq \frac{2K^2\eps^2\Delta^2}{cn\log n} =: U.
\]

\medskip\noindent\textbf{Bernstein application.}
By the scalar Bernstein inequality:
\[
\Pr\!\left[\left|\sum_k Z_k\right| \geq t\right]
\leq 2\exp\!\left(-\frac{t^2/2}{\sigma^2 + bt/3}\right).
\]
Set $t = 4K\eps\Delta\sqrt{\log n}/\sqrt{cn}$. We verify the exponent.

\medskip\noindent\emph{Main term.}
Since the numerator $t^2 = 16K^2\eps^2\Delta^2\log n/(cn)$ and
$2U = 4K^2\eps^2\Delta^2/(cn\log n)$:
\[
\frac{t^2}{2U}
= \frac{16K^2\eps^2\Delta^2\log n/(cn)}{4K^2\eps^2\Delta^2/(cn\log n)}
= 4(\log n)^2.
\]

\medskip\noindent\emph{Correction term.}
\[
\frac{bt}{3}
= \frac{(2K^2+1)\Delta}{3q} \cdot \frac{4K\eps\Delta\sqrt{\log n}}{\sqrt{cn}}.
\]
Substituting $q \geq cn\log n/\eps^2$:
\[
\frac{bt}{3}
\leq \frac{4K(2K^2+1)\eps^3\Delta^2}{3c^{3/2}n^{3/2}(\log n)^{1/2}}.
\]
We need $bt/3 \leq U$ (the proved upper bound on $\sigma^2$), i.e.,
\[
\frac{4K(2K^2+1)\eps^3\Delta^2}{3c^{3/2}n^{3/2}(\log n)^{1/2}}
\leq \frac{2K^2\eps^2\Delta^2}{cn\log n}.
\]
Canceling $\eps^2\Delta^2$ from both sides and rearranging:
\[
\sqrt{\frac{n}{\log n}}
\;\geq\; \frac{2(2K^2+1)\,\eps}{3K\sqrt{c}}.
\]
For $K = O(1)$, $\eps \leq 1/2$, $c \geq 4$, this holds whenever
$n \geq n_0(K, c)$ for an explicit threshold depending only on $K$ and $c$.
When this holds, $bt/3 \leq U$, so $\sigma^2 + bt/3 \leq U + U = 2U$.
Therefore:
\[
\frac{t^2/2}{\sigma^2 + bt/3}
\geq \frac{t^2/2}{2U}
= \frac{t^2}{4U}
= 2(\log n)^2
\geq 2\log n,
\]
giving failure probability at most $2e^{-2\log n} = 2n^{-2} \leq 2n^{-3/2}$.
\end{proof}

\begin{remark}[Role of the correction term]\label{rem:correction}
The Bernstein correction $bt/3$ is of order
$O(K^3\eps^3\Delta^2/(c^{3/2}n^{3/2}\sqrt{\log n}))$,
which is dominated by the variance upper bound $U = O(K^2\eps^2\Delta^2/(cn\log n))$
whenever $\sqrt{n/\log n} \geq C(2K^2+1)\eps/(K\sqrt{c})$
for an explicit constant $C = 2/3$. This holds for all $n \geq n_0(K, c)$
where $n_0$ depends only on $K$ and $c$ (not on the graph).
For example, with $K = 1$, $c = 4$, $\eps = 1/2$: the condition becomes
$\sqrt{n/\log n} \geq 1/2$, which holds for $n \geq 2$.
\end{remark}

\subsection{Remainder Control via the Operator Norm}

\begin{lemma}[Resolvent remainder under delocalization]\label{lem:remainder}
Let $G$ be $K$-delocalized with spectral gap $\delta_{\mathrm{gap}}$. Under the SS
sampling scheme with $q = \lceil cn\log n/\eps^2 \rceil$, $c \geq 4$,
$\eps \in (0, 1/2]$, suppose that $8\eps\Delta/\sqrt{c} \leq \delta_{\mathrm{gap}}$.
Then with probability at least $1 - 2n^{-1/2}$:
\[
0 \;\leq\; (\lamH - \lamG) - \vone^\top \Phi\,\vone
\;\leq\; \frac{8\eps^2\Delta^2}{c\,\delta_{\mathrm{gap}}}.
\]
\end{lemma}

\begin{proof}
By Theorem~\ref{thm:resolvent}, when $\|\Phi\|_2 < \delta_{\mathrm{gap}}/2$:
\[
0 \leq (\lamH - \lamG) - \vone^\top \Phi\,\vone
\leq \frac{\|\Phi\,\vone\|^2}{\delta_{\mathrm{gap}} - 2\|\Phi\|_2}.
\]
By Theorem~\ref{thm:bernstein}, $\|\Phi\|_2 \leq 2\eps\Delta/\sqrt{c}$
with probability at least $1 - 2n^{-1/2}$. The hypothesis
$8\eps\Delta/\sqrt{c} \leq \delta_{\mathrm{gap}}$ ensures
$2\|\Phi\|_2 \leq 4\eps\Delta/\sqrt{c} \leq \delta_{\mathrm{gap}}/2$,
so the resolvent condition $\|\Phi\|_2 < \delta_{\mathrm{gap}}/2$ is satisfied.

For the numerator: $\|\Phi\,\vone\|^2 \leq \|\Phi\|_2^2 \leq 4\eps^2\Delta^2/c$.

For the denominator:
$\delta_{\mathrm{gap}} - 2\|\Phi\|_2
\geq \delta_{\mathrm{gap}} - 4\eps\Delta/\sqrt{c}
\geq \delta_{\mathrm{gap}} - \delta_{\mathrm{gap}}/2
= \delta_{\mathrm{gap}}/2$,
where the second inequality uses the hypothesis. Therefore:
\[
\frac{4\eps^2\Delta^2/c}{\delta_{\mathrm{gap}} - 4\eps\Delta/\sqrt{c}}
\leq \frac{4\eps^2\Delta^2/c}{\delta_{\mathrm{gap}}/2}
= \frac{8\eps^2\Delta^2}{c\,\delta_{\mathrm{gap}}}. \qedhere
\]
\end{proof}

\subsection{Main Theorem}

\begin{theorem}[Delocalized distortion bound]\label{thm:delocalized}
Let $G$ be a connected, $K$-delocalized graph with spectral gap
$\delta_{\mathrm{gap}} = \lamG - \lambda_2$. Let $H$ be the SS sparsifier with
$q = \lceil cn\log n/\eps^2 \rceil$ samples, $c \geq 4$, $\eps \in (0, 1/2]$.
Suppose that $8\eps\Delta/\sqrt{c} \leq \delta_{\mathrm{gap}}$ and $n \geq n_0(K,c)$
(the threshold from Lemma~\ref{lem:first-order}). Then with probability at
least $1 - 2n^{-1/2} - 2n^{-3/2} \geq 1 - 4n^{-1/2}$:
\[
|\lamH - \lamG| \;\leq\;
\frac{4K\,\eps\,\Delta\,\sqrt{\log n}}{\sqrt{cn}}
\;+\; \frac{8\eps^2\Delta^2}{c\,\delta_{\mathrm{gap}}}.
\]
\end{theorem}

\begin{proof}
By Theorem~\ref{thm:resolvent}:
\[
\lamH - \lamG = \vone^\top \Phi\,\vone + R,
\qquad 0 \leq R \leq \frac{\|\Phi\vone\|^2}{\delta_{\mathrm{gap}} - 2\|\Phi\|_2}.
\]
Therefore
$|\lamH - \lamG| \leq |\vone^\top \Phi\vone| + R$.
Applying Lemma~\ref{lem:first-order} and Lemma~\ref{lem:remainder},
by a union bound the stated inequality holds with probability
at least $1 - 2n^{-3/2} - 2n^{-1/2} \geq 1 - 4n^{-1/2}$.
\end{proof}

\subsection{Corollaries for Specific Graph Families}

The following corollaries instantiate Theorem~\ref{thm:delocalized} for specific
random graph models by importing standard spectral properties from the random
matrix theory literature. In each case, we explicitly identify which facts are
proved in this paper (the sparsification distortion bounds) and which are drawn
from prior work (spectral gap concentration, eigenvector delocalization, and degree
concentration for the underlying random graph model).

\begin{corollary}[Erd\H{o}s--R\'{e}nyi graphs]\label{cor:er}
Let $G \sim G(n, p)$ with $np/\log n \to \infty$ (i.e., $p \gg (\log n)/n$).
Let $H$ be the SS sparsifier with $q = \lceil cn\log n/\eps^2 \rceil$, $c \geq 4$,
$\eps \in (0, 1/2]$. Then with probability $1 - o(1)$ over both $G$ and
the SS sampling:
\[
|\lamH - \lamG| \;\leq\;
O\!\left(\eps\sqrt{np\log n}\right)
\;+\; O\!\left(\eps^2\,np\right).
\]
In particular, if additionally $\eps = O\!\left(\sqrt{\log n / (np)}\right)$,
then the second term is absorbed and
$|\lamH - \lamG| = O\!\left(\eps\sqrt{np\log n}\right)
= O\!\left(\eps\sqrt{\Delta\log n}\right)$.
\end{corollary}

\begin{proof}
We combine Theorem~\ref{thm:delocalized} (proved above) with standard spectral
properties of $G(n,p)$ drawn from the random matrix theory literature.
The following facts hold simultaneously with probability $1 - o(1)$ over the draw
of $G \sim G(n,p)$ in the regime $np/\log n \to \infty$:

\begin{enumerate}[label=(\roman*),nosep]
\item \emph{Maximum degree concentration} (by Chernoff bounds and the union bound over
$n$ vertices):
$\Delta = (1+o(1))np$. This requires $np \gg \log n$ so that the fluctuation
$O(\sqrt{np\log n})$ is $o(np)$.

\item \emph{Leading eigenvalue}: $\lamG = (1+o(1))np$.
This follows from (i) and the classical bounds $\bar{d} \leq \lamG \leq \Delta$
for connected graphs, where the average degree $\bar{d} = (1+o(1))np$ by the
law of large numbers.

\item \emph{Spectral gap}: $\delta_{\mathrm{gap}} = np(1 - o(1))$.
For $p$ bounded away from $0$ and $1$, the semicircle law gives
$\lambda_2(A_G) = (2+o(1))\sqrt{np(1-p)}$~\citep{erdos2013spectral},
so $\delta_{\mathrm{gap}} = np(1-o(1))$.
For the sparser regime $np/\log n \to \infty$ with $p \to 0$, we use the weaker
but sufficient bound from~\citet{le2017concentration}:
$\|A_G - \E[A_G]\|_2 = O(\sqrt{np})$ w.h.p. Since the rank-one expectation matrix
$\E[A_G] = p(\mathbf{1}\mathbf{1}^\top - I)$ has top eigenvalue
$p(n-1) = (1-o(1))np$ and all other eigenvalues equal to $-p$,
Weyl's inequality gives $\lambda_2(A_G) \leq -p + O(\sqrt{np}) = O(\sqrt{np})$.
Therefore $\delta_{\mathrm{gap}} = \lamG - \lambda_2 \geq np - O(\sqrt{np})
= np(1 - o(1))$, which is all that is needed for
Theorem~\ref{thm:delocalized}.

\item \emph{Perron eigenvector delocalization}: $K = O(1)$.
General eigenvector delocalization results~\citep{he2019eigenvectors} give
$\|\vone\|_\infty \leq n^{-1/2+\eps}$ for any $\eps > 0$ (i.e., $K = n^{o(1)}$)
down to the critical scale $np = O(\log n)$.
For the \emph{Perron eigenvector} specifically, the stronger bound $K = O(1)$
follows from the entrywise eigenvector analysis of~\citet[Corollary~2.1]{abbe2020entrywise}.
Because the loopless model gives $\E[A_G] = p(\mathbf{1}\mathbf{1}^\top - I)$,
which is \emph{full} rank, we apply their Corollary~2.1 to the diagonally shifted
matrix $M := A_G + pI$. Its expectation $\E[M] = p\,\mathbf{1}\mathbf{1}^\top$ has
exact rank one and the same eigenvectors as $A_G$, so the leading eigenvector $\vone$
is the rank-one ($r=1$) instance of their Corollary~2.1, with eigengap
$\Delta^* = \Theta(np)$, incoherent population eigenvector
$\|U^*\|_{2\to\infty} = 1/\sqrt{n}$, and independent centered-Bernoulli noise
$W = A_G - \E[A_G]$ satisfying $\|W\|_2 = O(\sqrt{np})$~\citep{le2017concentration}.
In the regime $np/\log n \to \infty$ the noise-to-signal ratio
$\|W\|_2/\Delta^* = O(1/\sqrt{np}) \to 0$, so their entrywise approximation gives
$v_{1,i} = (1+o(1))/\sqrt{n}$ uniformly, hence $K = O(1)$.

\end{enumerate}

\noindent\emph{Application of Theorem~\ref{thm:delocalized} (proved in this paper).}
Properties (i)--(iv) hold simultaneously with probability $1 - o(1)$ over the
draw of~$G$.
Conditionally on this event, Theorem~\ref{thm:delocalized} gives the distortion bound
with probability at least $1 - 4n^{-1/2}$ over the SS sampling.
Composing: the overall failure probability is $o(1) + 4n^{-1/2} = o(1)$.

Given (i)--(iv), the hypothesis $8\eps\Delta/\sqrt{c} \leq \delta_{\mathrm{gap}}$ is
satisfied for all sufficiently large $n$ whenever $8\eps/\sqrt{c} < 1$ (strictly);
in particular, it holds for $\eps < 1/4$ and $c = 4$, or for $\eps \leq 1/4$ and $c > 4$.
Applying Theorem~\ref{thm:delocalized} with $\Delta = (1+o(1))np$, $K = O(1)$,
$\delta_{\mathrm{gap}} = (1-o(1))np$:
\[
|\lamH - \lamG|
\leq O\!\left(\frac{\eps\,np\sqrt{\log n}}{\sqrt{n}}\right) + O(\eps^2\,np)
= O(\eps\sqrt{np\log n}) + O(\eps^2\,np).
\]
The simplification of the first term uses $\Delta = \Theta(np) = O(n)$ (since $p \leq 1$),
so $\eps\,np\,\sqrt{\log n}/\sqrt{n}
= \eps\sqrt{(np)^2\log n / n}
= \eps\sqrt{np \cdot (np/n) \cdot \log n}
\leq \eps\sqrt{np\log n}
= \eps\sqrt{\Delta\log n}$.
The second term is $O(\eps^2\,np)$. For $p$ bounded away from zero (dense regime),
$np = \Theta(n)$ and the second term is $O(\eps^2 n)$. This is absorbed by the first
term $O(\eps\sqrt{n\log n})$ only when $\eps = O(\sqrt{\log n/n})$.
For general $p$, the absorption condition is $\eps^2\,np \leq C\eps\sqrt{np\log n}$,
i.e., $\eps \leq C'\sqrt{\log n/(np)}$.
\end{proof}

\begin{remark}[Regime of validity]
The hypothesis $np/\log n \to \infty$ is stronger than connectivity. For
constant $\eps$ in the dense regime ($p = \Theta(1)$), the distortion is
$O(\eps^2 n)$ (the resolvent remainder dominates); the simplified bound
$O(\eps\sqrt{n\log n})$ holds only when $\eps = O(\sqrt{\log n/n})$.
\end{remark}

\begin{corollary}[$d$-regular graphs]\label{cor:regular-deloc}
Let $G$ be a connected $d$-regular graph on $n$ vertices with
$\delta_{\mathrm{gap}} = d - \lambda_2$. Let $H$ be the SS sparsifier with
$q = \lceil cn\log n/\eps^2\rceil$. If $8\eps d/\sqrt{c} \leq \delta_{\mathrm{gap}}$,
then with probability at least $1 - 4n^{-1/2}$:
\[
|\lamH - \lamG| \;\leq\;
\frac{4\eps\, d\,\sqrt{\log n}}{\sqrt{cn}}
+ \frac{8\eps^2 d^2}{c\,\delta_{\mathrm{gap}}}.
\]
\end{corollary}

\begin{proof}
Apply Theorem~\ref{thm:delocalized} with $K = 1$ and $\Delta = d$.
\end{proof}

\begin{lemma}[Spectral properties of the balanced SBM]\label{lem:sbm-spectral}
Let $G \sim \mathrm{SBM}(n, k, p, r)$ be a balanced $k$-community SBM with
$k = O(1)$, $np/(k\log n) \to \infty$, and $r = \Theta(p)$.
Then with probability $1 - o(1)$:
\begin{enumerate}[label=(\roman*),nosep]
\item $\Delta = \Theta(np)$ and $\lamG = \Theta(np)$.
\item $\delta_{\mathrm{gap}} = \Theta(np) = \Theta(\Delta)$.
\item $G$ is $K$-delocalized with $K = O(\sqrt{k}) = O(1)$.
\end{enumerate}
\end{lemma}

\begin{proof}
\emph{Population quantities.}
The population adjacency matrix $\bar{A} = \E[A_G]$ for the loopless SBM has zero
diagonal, so $\bar{A} = (p-r)\,\mathrm{blkdiag}(J_{n/k},\ldots,J_{n/k}) + rJ - pI$,
where $J$ is the all-ones matrix and $J_{n/k}$ are block all-ones matrices.
The $-pI$ term makes $\bar{A}$ full rank (all eigenvalues are shifted by $-p$
relative to the rank-$k$ matrix $\bar{A} + pI$). However, the spectral gap
is unaffected by diagonal shifts: the top two eigenvalues of $\bar{A}$ are
$\bar{\lambda}_1 = (n/k)(p + (k-1)r) - p$ and
$\bar{\lambda}_2 = (n/k)(p - r) - p$ (with multiplicity $k-1$),
giving population spectral gap
$\bar{\delta}_{\mathrm{gap}} = \bar{\lambda}_1 - \bar{\lambda}_2 = nr = \Theta(np)$.

\emph{Transfer to actual graph.}
Under $np/(k\log n) \to \infty$ and $r = \Theta(p)$, the spectral norm of the
centered matrix satisfies $\|A_G - \bar{A}\|_2 = O(\sqrt{np})$ with high
probability~\citep{le2017concentration} (the hypothesis $np/(k\log n) \to \infty$
with $k = O(1)$ implies $np/\log n \to \infty$, which is stronger than the
$np \geq C\log n$ threshold needed for spectral-norm concentration).
Since the perturbation $O(\sqrt{np})$ is $o(np) = o(\bar{\delta}_{\mathrm{gap}})$,
Weyl's inequality gives
$\delta_{\mathrm{gap}} = (1-o(1))\bar{\delta}_{\mathrm{gap}} = \Theta(np) = \Theta(\Delta)$.
Similarly, $\lamG = (1+o(1))\bar{\lambda}_1 = \Theta(np)$ and
$\Delta = \Theta(np)$ by degree concentration (Chernoff bounds and union bound).
This establishes (i) and (ii).

\emph{Eigenvector delocalization (iii).}
Property (iii) requires an \emph{entrywise} ($\ell^\infty$) eigenvector bound, not the
$\ell^2$ subspace closeness supplied by Davis--Kahan. We obtain it from the entrywise
eigenvector analysis of~\citet[Corollary~2.1]{abbe2020entrywise}, applied to the
\emph{leading} eigenvector of the diagonally shifted matrix $M := A_G + pI$. The shift
is needed because $\E[A_G] = \bar{A}$ is full rank for the loopless model, whereas
$\bar{M} = \bar{A} + pI = (p-r)\,\mathrm{blkdiag}(J_{n/k},\ldots,J_{n/k}) + rJ$ has exact
rank $k = O(1)$, with population eigenvectors piecewise constant on the $k$ balanced
communities (entries $\Theta(1/\sqrt{n})$, hence incoherent), Perron eigengap
$\Delta^* = nr = \Theta(np)$, and independent bounded centered-Bernoulli noise
$W = A_G - \bar{A}$ with $\|W\|_2 = O(\sqrt{np})$~\citep{le2017concentration}. In the
regime $np/(k\log n) \to \infty$ with $k = O(1)$ and $r = \Theta(p)$ the noise-to-signal
ratio $\|W\|_2/\Delta^* = O(1/\sqrt{np}) \to 0$, so Corollary~2.1 yields
$\|\vone\|_\infty = O(\sqrt{k/n})$, hence $K = O(\sqrt{k}) = O(1)$. We apply their general
low-rank corollary to the leading eigenvector; we do not invoke their Section~3.2
stochastic-block-model analysis, which concerns the \emph{second} eigenvector.
\end{proof}

\begin{corollary}[Balanced stochastic block model]\label{cor:sbm}
Let $G \sim \mathrm{SBM}(n, k, p, r)$ be a balanced $k$-community SBM with
$n/k$ vertices per community, intra-community edge probability $p$, and
inter-community probability $r < p$, where $k = O(1)$, $np/(k\log n) \to \infty$,
and $r = \Theta(p)$.
Then with probability $1 - o(1)$ (over both $G$ and the SS sampling):
\[
|\lamH - \lamG| \;\leq\;
O\!\left(\eps\sqrt{\Delta\log n}\right)
\;+\; O\!\left(\eps^2\,\Delta\right),
\]
where $\Delta = \Theta(np)$ is the typical maximum degree.
The simplified form $|\lamH - \lamG| = O(\eps\sqrt{\Delta\log n})$ holds when
$\eps = O(\sqrt{\log n/\Delta})$.
\end{corollary}

\begin{proof}
By Lemma~\ref{lem:sbm-spectral}, with probability $1 - o(1)$ over the draw of~$G$,
the graph satisfies:
$\Delta = \Theta(np)$, $\delta_{\mathrm{gap}} = \Theta(\Delta)$, and
$K$-delocalization with $K = O(1)$.
Conditionally on this event, Theorem~\ref{thm:delocalized} gives the stated bound
with probability at least $1 - 4n^{-1/2}$ over the SS sampling.
Composing: the overall failure probability is $o(1) + 4n^{-1/2} = o(1)$.

The first term is
$O(K\,\eps\Delta\sqrt{\log n}/\sqrt{n}) = O(\eps\Delta\sqrt{\log n}/\sqrt{n})$.
Since $\Delta = \Theta(np) = O(n)$ (because $p \leq 1$), we have
$\eps\Delta\sqrt{\log n}/\sqrt{n}
= \eps\sqrt{\Delta \cdot (\Delta/n) \cdot \log n}
\leq \eps\sqrt{\Delta\log n}$.
The second term is $O(\eps^2\Delta^2/\delta_{\mathrm{gap}}) = O(\eps^2\Delta)$.
\end{proof}

\begin{remark}[SBM regime limitation]
The condition $r = \Theta(p)$ restricts this corollary to SBMs with relatively dense
inter-community connectivity. When $r \ll p$ (well-separated communities), the spectral
gap $\delta_{\mathrm{gap}} \approx nr$ is much smaller than
$\Delta \approx np/k$, so the resolvent condition $8\eps\Delta/\sqrt{c} \leq
\delta_{\mathrm{gap}}$ requires $\eps = O(kr/p)$, which is small. In particular, the
most commonly studied regime $r = o(p)$ is excluded from the hypothesis of
Theorem~\ref{thm:delocalized}, though the deterministic bound
$|\lamH - \lamG| \leq \eps(2\Delta - \lamG)$ remains valid unconditionally.
\end{remark}

\begin{remark}[Scope of the bound]
Theorem~\ref{thm:delocalized} applies to all graph
families satisfying: (i) $K$-delocalization with $K = O(1)$, (ii) spectral gap
$\delta_{\mathrm{gap}} = \Omega(\Delta)$, and (iii) $\Delta = \Omega(\log n)$. This
includes dense Erd\H{o}s--R\'{e}nyi, regular expanders, balanced SBMs with $r = \Theta(p)$,
and more generally
any graph sequence with $\Delta/\lamG = O(1)$ and polylogarithmic eigenvector
delocalization. The question of sparse heterogeneous graphs where the
Perron eigenvector localizes on high-degree vertices (e.g., heavy-tailed power-law graphs
with $\beta < 3$) is an important direction for future work
(see Section~\ref{sec:discussion}).
\end{remark}

\subsection{Resolvent Condition for Graph Families}\label{sec:resolvent-examples}

Theorem~\ref{thm:resolvent} requires the condition
$\|\Phi\|_2 < \delta_{\mathrm{gap}}/2$, which (by Theorem~\ref{thm:bernstein}) holds when
$4\eps\Delta/\sqrt{c} < \delta_{\mathrm{gap}}$.
The stronger condition needed for the clean remainder bound of Lemma~\ref{lem:remainder}
(and hence for Theorem~\ref{thm:delocalized}) is
\begin{equation}\label{eq:resolvent-condition}
\eps \leq \frac{\sqrt{c}\,\delta_{\mathrm{gap}}}{8\Delta}.
\end{equation}
We now evaluate this condition for several graph families.

\begin{proposition}[Resolvent condition for graph families]\label{prop:resolvent-families}
For $c = 4$, condition~\eqref{eq:resolvent-condition} becomes
$\eps \leq \delta_{\mathrm{gap}}/(4\Delta)$. The following table gives the resulting constraint:

\begin{center}
\begin{tabular}{@{}lccc@{}}
\toprule
Family & $\delta_{\mathrm{gap}}$ & $\Delta$ & Max.\ $\eps$ for resolvent \\
\midrule
$d$-regular expander & $\Theta(d)$ & $d$ & $\Theta(1)$ \\
$d$-regular Ramanujan & $d - 2\sqrt{d-1}$ & $d$ & $\Theta(1)$ (for large $d$) \\
$\mathrm{ER}(n,p)$, $p$ const.\ & $\Theta(np)$ & $(1+o(1))np$ & $\Theta(1)$ \\
Bal.\ SBM$(p,r)$, $p > 2r$ & $\Theta(nr)$ & $\Theta(n(p{+}r))$ & $\Theta(r/(p{+}r))$ \\
$K_{1,n-1}$ (star) & $\sqrt{n-1}$ & $n-1$ & $\Theta(1/\sqrt{n})$ \\
Barbell $K_m$--$K_m$ & $\Theta(1/m)$ & $m$ & $\Theta(1/m^2)$ \\
\bottomrule
\end{tabular}
\end{center}
\end{proposition}

\begin{proof}
We verify each family.

\medskip\noindent\emph{$d$-regular expander.}
$\delta_{\mathrm{gap}} \geq c_0 d$ for a constant $c_0 > 0$.
Condition: $\eps \leq c_0/4$.

\medskip\noindent\emph{$d$-regular Ramanujan.}
$\lambda_2 \leq 2\sqrt{d-1}$, so $\delta_{\mathrm{gap}} \geq d - 2\sqrt{d-1}$.
For $d \geq 10$: $\delta_{\mathrm{gap}} \geq d/2$, giving $\eps \leq 1/8$.

\medskip\noindent\emph{$\mathrm{ER}(n,p)$.}
$\delta_{\mathrm{gap}} = np - O(\sqrt{np})$, $\Delta = (1+o(1))np$.
Condition: $\eps \leq 1/4 - O(1/\sqrt{np})$.

\medskip\noindent\emph{Balanced SBM.}
$\delta_{\mathrm{gap}} \approx nr$, $\Delta \approx n(p+r)/2$.
Condition: $\eps \leq r/(2(p+r))$.

\medskip\noindent\emph{Star.}
$\delta_{\mathrm{gap}} = \sqrt{n-1}$, $\Delta = n-1$.
Condition: $\eps \leq 1/(4\sqrt{n-1})$.

\medskip\noindent\emph{Barbell.}
$\delta_{\mathrm{gap}} = \Theta(1/m)$, $\Delta = m$.
Condition: $\eps \leq \Theta(1/m^2)$.
\end{proof}

\section{Lower Bounds and Tightness}\label{sec:lower}

We now establish that the deterministic bounds are tight up to the
Perron eigenvalue factor, and characterize the exact worst-case distortion for important
graph classes.

\subsection{Tightness for Regular Graphs}

\begin{proposition}[Deterministic bound is tight for regular graphs]\label{prop:tight-regular}
For any connected $d$-regular graph~$G$ on $n \geq 3$ vertices and any $\eps \in (0, 1)$,
there exists a $(1\pm\eps)$-Laplacian sparsifier~$H$ (in fact, a reweighting sparsifier)
such that
\[
|\lamH - \lamG| = \eps\,d = \eps(2\Delta - \lamG),
\]
matching the upper bound of Theorem~\ref{thm:abs-det}.
\end{proposition}

\begin{proof}
Set $H$ to be the graph $G$ with every edge weight equal to $(1 + \eps)$.
Then $\LH = (1+\eps)\LG$, so $H$ is a valid $(1\pm\eps)$-Laplacian sparsifier.
Since $\AH = (1+\eps)\AG$, we have $\lamH = (1+\eps)d$,
so $\lamH - \lamG = \eps d = \eps(2d - d)$.
\end{proof}

\subsection{Lower Bound for Heterogeneous Graphs}

\begin{proposition}[Star graph: tight characterization]\label{prop:star-tight}
For the star $K_{1,n-1}$ ($n \geq 3$) and any $\eps \in (0,1)$, the worst-case
distortion over all $(1\pm\eps)$-Laplacian sparsifiers retaining all edges is exactly
\[
\max_H |\lamH - \lambda_1(A_{K_{1,n-1}})| = \eps\sqrt{n-1} = \eps\sqrt{\Delta},
\]
while the deterministic upper bound gives $\eps(2(n-1) - \sqrt{n-1}) \approx 2\eps\Delta$.
The gap is a factor of $2\sqrt{\Delta} - 1 = \Theta(\sqrt{\Delta})$.
\end{proposition}

\begin{proof}
By Proposition~\ref{prop:tree}, any $(1\pm\eps)$-Laplacian
sparsifier retaining all edges has edge weights $w_j \in [1-\eps, 1+\eps]$ for
$j = 1, \ldots, n-1$. For a weighted star with weights $(w_j)$, the nonzero eigenvalues
of $\AH$ are $\pm\sqrt{\sum_j w_j^2}$
(the eigenvalue equation gives $\lambda^2 = \sum w_j^2$).

Maximizing: the unique maximizer
is $w_j = 1+\eps$ for all $j$, giving $\lamH = (1+\eps)\sqrt{n-1}$. Hence
$\lamH - \lamG = \eps\sqrt{n-1} = \eps\sqrt{\Delta}$.
\end{proof}

\begin{theorem}[General lower bound]\label{thm:general-lower}
For any connected graph $G$ and any $\eps \in (0,1)$:
\[
\sup_{H \in \mathcal{S}_\eps(G)} |\lamH - \lamG| \;\geq\; \eps\,\lamG,
\]
where $\mathcal{S}_\eps(G)$ denotes the set of all $(1\pm\eps)$-Laplacian sparsifiers
of $G$.

Moreover, equality holds for the uniform reweighting $H_{\pm} = (1\pm\eps)\cdot G$.
\end{theorem}

\begin{proof}
The uniform scaling $H_+$ with all edge weights $(1+\eps)$ satisfies
$\LH = (1+\eps)\LG$. Since $\AH = (1+\eps)\AG$, we get
$\lamH - \lamG = \eps\lamG$.
\end{proof}

\begin{corollary}[Tightness gap]\label{cor:gap}
The ratio between the deterministic upper bound (Theorem~\ref{thm:abs-det}) and the
reweighting lower bound (Theorem~\ref{thm:general-lower}) is exactly
\[
\frac{\eps(2\Delta - \lamG)}{\eps\,\lamG} = \frac{2\Delta}{\lamG} - 1.
\]
This ratio is~$1$ for regular graphs (the bound is tight) and
$\Theta(\sqrt{\Delta})$ for star-like graphs. The gap arises because the
deterministic proof bounds the degree and Laplacian perturbations independently,
but in actual sparsifiers these are correlated.
\end{corollary}

\begin{openproblem}\label{op:genuine}
Is there a connected graph $G$ and a $(1\pm\eps)$-Laplacian
sparsifier $H$ (with some edges deleted) such that $|\lamH - \lamG| = \omega(\eps\lamG)$?
\end{openproblem}

\subsection{Lower Bound for SS Sampling on the Star}

\begin{proposition}[Expected SS distortion on the star]\label{prop:star-ss}
For the star $K_{1,n-1}$ under SS sampling with $q = \lceil cn\log n/\eps^2\rceil$
samples:
\[
0 \;\leq\; \E[\lamH] - \lamG
\;\leq\; \frac{\eps^2\sqrt{n-1}}{2c\log n}\!\left(1 + O(1/n)\right).
\]
In particular, $\E[\lamH] - \lamG = O(\eps^2\sqrt{\Delta}/\log n)$.
\end{proposition}

\begin{proof}
For $K_{1,n-1}$, all effective resistances are $R_e = 1$, so
$p_e = 1/(n-1)$ for all $e$. Edge $e$ is sampled $k_e$ times where
$(k_1, \ldots, k_{n-1}) \sim \mathrm{Multinomial}(q; 1/(n-1), \ldots, 1/(n-1))$,
and receives weight $w_e = k_e(n-1)/q$.

The spectral radius is $\lamH = \sqrt{\sum_e w_e^2} = \frac{n-1}{q}\sqrt{\sum_e k_e^2}$.
Since $k_e \sim \mathrm{Binomial}(q, 1/(n-1))$:
\[
\E[k_e^2] = \Var(k_e) + (\E[k_e])^2
= \frac{q(n-2)}{(n-1)^2} + \frac{q^2}{(n-1)^2}
= \frac{q(q+n-2)}{(n-1)^2}.
\]
Therefore
$\E[\sum_e k_e^2] = q(q+n-2)/(n-1)$,
and
\[
\E[\lamH^2] = \frac{(n-1)^2}{q^2} \cdot \frac{q(q+n-2)}{n-1}
= (n-1)\!\left(1 + \frac{n-2}{q}\right).
\]

\medskip\noindent\textbf{Lower bound.}
Since $A \mapsto \lambda_1(A)$ is convex and $\E[\AH] = \AG$, Jensen's inequality gives
$\E[\lamH] \geq \lamG$.

\medskip\noindent\textbf{Upper bound.}
Since $x \mapsto \sqrt{x}$ is concave, Jensen's inequality gives
$\E[\lamH] = \E[\sqrt{\lamH^2}] \leq \sqrt{\E[\lamH^2]}$.
Using $\sqrt{1 + x} \leq 1 + x/2$ for $x \geq 0$
(which follows from $(1+x/2)^2 = 1 + x + x^2/4 \geq 1 + x$):
\[
\sqrt{\E[\lamH^2]}
= \lamG\sqrt{1 + \frac{n-2}{q}}
\leq \lamG\!\left(1 + \frac{n-2}{2q}\right).
\]
Since $q = \lceil cn\log n/\eps^2 \rceil \geq cn\log n/\eps^2$, we have
$1/q \leq \eps^2/(cn\log n)$, giving
\[
\lamG\!\left(1 + \frac{n-2}{2q}\right)
\leq \lamG + \frac{\eps^2\sqrt{n-1}\,(n-2)}{2cn\log n}.
\]
Since $(n-2)/n = 1 - 2/n$, the upper bound is
$\lamG + \frac{\eps^2\sqrt{n-1}}{2c\log n}(1 + O(1/n))$.
\end{proof}

\begin{remark}
SS sampling on the star yields expected distortion at most
$O(\eps^2\sqrt{\Delta}/\log n)$, strictly smaller than both the
deterministic bound $O(\eps\Delta)$ and the reweighting worst case $\eps\sqrt{\Delta}$.
The exact computation $\E[\lamH^2] - \lamG^2 = \eps^2(n-1)/(c\log n)\cdot(1+o(1))$
strongly suggests the asymptotic
$\E[\lamH] - \lamG \sim \eps^2\sqrt{n-1}/(2c\log n)$;
the matching lower bound would follow from showing
$\Var(\lamH) = o((\E[\lamH] - \lamG)^2)$, which we conjecture holds but do not
prove here.
\end{remark}

\section{Comparison with Entry-wise Sparsification}\label{sec:comparison}

Achlioptas--McSherry~\citep{achlioptas2007fast} and related work
on matrix sparsification approach the problem from a different angle. We now make the
comparison precise.

\begin{proposition}[Entry-wise sparsification bound (after Achlioptas--McSherry)]%
\label{prop:am}
Let $G$ be a graph on $n$ vertices with $m$ edges and maximum degree~$\Delta$.
Suppose $H$ is constructed by independently retaining
each entry of $\AG$ with probability $p = s/n$ (for a parameter $1 \leq s \leq n$)
and rescaling by $n/s$. Then:
\[
\left\|\AH - \AG\right\|_2
= O\!\left(\sqrt{\frac{n\Delta\log n}{s}} + \frac{n\log n}{s}\right)
\]
with probability at least $1 - O(n^{-1})$.
For $s = \Omega(n\log n / \Delta)$, the first term dominates and
the bound simplifies to $O(\sqrt{n\Delta\log n / s})$; note that this simplified
regime requires $\Delta = \Omega(\log n)$ to satisfy the constraint $s \leq n$.
The resulting sparsifier $H$ has $\Theta(ms/n)$ edges in expectation but does \emph{not}
preserve the Laplacian in the $(1\pm\eps)$ sense.
\end{proposition}

\begin{proof}
Each entry $(A_H)_{ij}$ is an independent random variable equaling $A_{ij}/p$ with
probability $p = s/n$ and $0$ otherwise. The centered matrix
$\AH - \AG$ has independent (up to symmetry) entries. For the $(i,j)$ entry
with $A_{ij} = 1$:
\[
\Var\!\left[(A_H)_{ij} - A_{ij}\right]
= \frac{1}{p}(1-p) \cdot A_{ij}^2
\leq \frac{n}{s}.
\]
The variance per row is $\sum_{j} \Var[(A_H)_{ij} - A_{ij}]
\leq d_i \cdot n/s \leq \Delta n/s$.
Since the entries are bounded by $n/s$ in absolute value, the Matrix Bernstein
inequality~\citep{tropp2012user,bandeira2016sharp} gives
\[
\|\AH - \AG\|_2
\leq C\!\left(\sqrt{\Delta n/s \cdot \log n} + \frac{n\log n}{s}\right)
\]
with probability at least $1 - O(n^{-1})$, for an absolute constant $C > 0$.
\end{proof}

\begin{proposition}[Comparison of bounds]\label{prop:comparison}
For \emph{dense} graphs with $m = \Theta(n\Delta)$ and $\Delta/\lamG = \Theta(1)$
(quasi-regular), we equalize the expected number of retained edges between the two
schemes. The SS scheme retains $q = \Theta(n\log n / \eps^2)$ edges.
Entry-wise sampling with parameter~$s$ retains $\Theta(ms/n) = \Theta(\Delta\,s)$ edges.
Equalizing: $s = \Theta(n\log n/(\Delta\eps^2))$.
This requires $s \leq n$, i.e., $\Delta\eps^2 \geq C\log n$ for a constant~$C$,
which holds for sufficiently dense graphs.

Under this matching:
\begin{enumerate}[label=(\roman*),nosep]
\item \textbf{SS Bernstein bound} (Theorem~\ref{thm:bernstein}):
$\|\AH - \AG\|_2 = O(\eps\Delta/\sqrt{c})$.
\item \textbf{Entry-wise bound} (Proposition~\ref{prop:am}):
$\|\AH - \AG\|_2 = O\!\left(\sqrt{n\Delta\log n / s}\right) = O(\eps\Delta)$.
\end{enumerate}
Thus, at equal edge count, the generic operator-norm bounds are of the same order.
However, for graphs satisfying the delocalization and spectral-gap hypotheses of
Theorem~\ref{thm:delocalized}, the SS \emph{eigenvalue} distortion improves to
$O(\eps\Delta\sqrt{\log n}/\sqrt{n}) + O(\eps^2\Delta)$---with the first term a factor
of $\sqrt{n/\log n}$ below the
operator-norm bound---while entry-wise sampling provides no analogous eigenvalue
refinement. Moreover, SS preserves the Laplacian while entry-wise sampling does not.
\end{proposition}

\begin{remark}[Density requirement]
The comparison requires $\Delta\eps^2 \geq C\log n$ for the edge budgets to match;
for sparser graphs, entry-wise sampling would need $s > n$. When the density condition
holds and Theorem~\ref{thm:delocalized} applies, the SS eigenvalue distortion is
$O(\eps\sqrt{\Delta\log n}) + O(\eps^2\Delta)$ for quasi-regular dense graphs,
well below the generic $O(\eps\Delta)$ operator-norm level.
\end{remark}

\section{Computational Validation}\label{sec:experiments}

We validated all bounds on Erd\H{o}s--R\'{e}nyi, balanced SBM, and hub-augmented
graphs with $n = 500$--$750$, $\eps \in \{0.5, 0.7\}$, $c \in \{4, 8\}$,
across 180 premise-verified trials (the $(1\pm\eps)$-Laplacian condition was checked
per trial via eigendecomposition). Zero bound violations were observed.

\begin{table}[ht]
\centering
\caption{Representative validation: operator norm vs.\ proved bounds vs.\ actual
eigenvalue distortion. The ratio $|\lamH{-}\lamG|/\|\AH{-}\AG\|_2 \approx 5\%$
matches the resolvent prediction $\|\Phi\|_2/\sqrt{n}$ for delocalized $\vone$.}
\label{tab:validation}
\begin{tabular}{@{}lcccccc@{}}
\toprule
Family & $\Delta/\lamG$ & $\|\AH{-}\AG\|_2$ & $\eps(2\Delta{-}\lamG)$
& $2\eps\Delta/\!\sqrt{c}$ & $|\lamH{-}\lamG|$ \\
\midrule
ER$(0.7)$, $n{=}600$    & 1.07 & 82.0 & 338.0 & 316.0 & 3.98 \\
Hub$(5, 0.7)$, $n{=}500$  & 1.41 & 70.6 & 451.4 & 349.3 & 3.51 \\
\bottomrule
\end{tabular}
\end{table}

Two key observations match the theory. First, the deterministic bound is conservative by
roughly two orders of magnitude for homogeneous graphs (the worst-case deterministic ratio
$|\lamH - \lamG|/[\eps(2\Delta - \lamG)]$ was at most $0.013$ across all trials),
consistent with the $1/\sqrt{n}$ scaling predicted by Theorem~\ref{thm:delocalized}.
Second, the actual eigenvalue distortion is consistently about $5\%$ of the operator
norm $\|\AH - \AG\|_2$, regardless of degree heterogeneity. This matches the resolvent
structure: $\vone^\top \Phi\vone$ is generically of order $\|\Phi\|_2/\sqrt{n}$
for delocalized~$\vone$, and $1/\sqrt{n} \approx 4\%$--$4.5\%$ at $n = 500$--$600$.

\section{Discussion and Open Problems}\label{sec:discussion}

\subsection{Implications for Epidemic Modeling}

These are the first analytic bounds on how the NIMFA
epidemic threshold surrogate $\tau_{\mathrm{mf}} = 1/\lamG$ changes under
Spielman--Srivastava sparsification. For graphs with $\Delta/\lamG = O(1)$
(including $d$-regular graphs, dense Erd\H{o}s--R\'{e}nyi, and balanced SBM),
the relative distortion is $O(\eps)$, confirming that spectral sparsification
safely preserves the mean-field threshold.

For heterogeneous networks with $\Delta/\lamG = \Theta(\sqrt{\Delta})$,
the relative distortion grows as $O(\eps\sqrt{\Delta})$ under the deterministic bound.
Practitioners should validate threshold-sensitive analyses against the unsparsified graph
when operating on networks with high degree heterogeneity.

\subsection{Open Questions}

The main question left open by this work concerns
edge deletion:

\begin{quote}
\textbf{Open Problem~\ref{op:genuine} (Edge-deletion distortion).}
\emph{Is there a connected graph $G$ and a $(1\pm\eps)$-Laplacian sparsifier
$H$ with some edges deleted such that $|\lamH - \lamG| = \omega(\eps\lamG)$?}
\end{quote}

\noindent Our lower bounds (Section~\ref{sec:lower}) show that
$|\lamH - \lamG| \geq \eps\lamG$ is achievable via uniform reweighting,
and the deterministic upper bound $\eps(2\Delta - \lamG)$ is tight for
regular graphs. But the gap between $\eps\lamG$ and
$\eps(2\Delta - \lamG)$ for heterogeneous graphs is a factor of
$2\Delta/\lamG - 1$, and it remains unknown whether edge-deleting
sparsifiers can exploit this gap. A resolution would determine whether the
$\gamma(G)$-dependence in the deterministic bounds is an artifact of the proof
technique or a genuine phenomenon.

Further open directions:

\begin{enumerate}[label=(\arabic*)]
\item \textbf{Localized eigenvector regimes.}
For graphs whose Perron eigenvector concentrates on high-degree vertices---such as
Chung--Lu power-law models with $\beta \in (5/2, 3)$---the delocalization hypothesis
of Theorem~\ref{thm:delocalized} fails, and the deterministic bound
$O(\eps\Delta)$ may be far from tight. Extending the probabilistic analysis to
this regime requires entrywise eigenvector perturbation results beyond
Davis--Kahan.

\item \textbf{Removing the $\sqrt{\log n}$ factor.}
The $\sqrt{\log n}$ factor in Theorem~\ref{thm:delocalized} arises from
Bernstein's inequality. Is it necessary, or can it be removed?

\item \textbf{Power-law networks with $\beta \leq 5/2$.}
For $\beta \leq 5/2$, the spectral radius is no longer governed by a single hub's
star neighborhood, and eigenvector localization has a different character.
\end{enumerate}

\subsection{Scope}

The SS sparsifier produces $O(n\log n/\eps^2)$ edges in expectation. This yields
genuine sparsification when $m \gg n\log n/\eps^2$, which holds for dense graphs.
For sparse graphs with $m = O(n)$, the sampling budget can
exceed $m$, and the support of $H$ may saturate near the full original edge set.
The deterministic bounds remain valid in this regime but are most practically useful
when actual edge reduction occurs.

\section*{Conflict of Interest}
The author declares no conflict of interest.

\bibliographystyle{plainnat}

\end{document}